\documentclass[a4paper,12pt]{article}
\usepackage{kotex}
\usepackage{enumerate}
\usepackage{amsmath,amssymb,amsthm, mathrsfs}
\numberwithin{equation}{section} 
\makeatletter
\renewcommand*\env@matrix[1][*\c@MaxMatrixCols c]{%
  \hskip -\arraycolsep
  \let\@ifnextchar\new@ifnextchar
  \array{#1}}
\makeatother
\usepackage{latexsym}
\usepackage{graphics}
\usepackage{hyperref}
\usepackage{graphicx}
\usepackage[bf, small]{titlesec}
\usepackage{authblk} 
\usepackage{soul}

\usepackage{lineno}

\theoremstyle{plain}
\newtheorem{Thm}{Theorem}[section]
\newtheorem{Lem}[Thm]{Lemma}
\newtheorem{Prop}[Thm]{Proposition}
\newtheorem{Cor}[Thm]{Corollary}
\newtheorem*{Claim*}{Claim}

\theoremstyle{definition}

\newtheorem{ex}[Thm]{Example}
\newtheorem{rem}[Thm]{Remark}
\usepackage{standalone}
\usepackage{pgfpages,tikz,pgfkeys,pgfplots}
\usetikzlibrary{arrows,positioning,matrix,fit,backgrounds,shapes,shapes.geometric, intersections}

\usetikzlibrary{shapes.callouts,decorations.pathmorphing} 
\usetikzlibrary{shadows} 
\usepackage{calc} 
\usetikzlibrary{decorations.markings} 
\tikzstyle{vertex}=[circle, draw, inner sep=0pt, minimum size=6pt] 

\usetikzlibrary{arrows,matrix} 

\definecolor{LemonChiffon}{rgb}{100, 98, 80}
\definecolor{myblue}{rgb}{0,0.4,0.8}
\definecolor{orange}{rgb}{1, 0.4, 0}
\definecolor{mygreen}{rgb}{0, 0.8, 0.4}
\definecolor{myred}{rgb}{204, 0, 0}
\definecolor{violet}{RGB}{0.4,0.2,1}
\definecolor{brown}{rgb}{0.6, 0.4, 0}

\title{Digraphs whose $m$-step competition graphs are trees}

\author[]{Myungho Choi
}
\author[]{Suh-Ryung Kim
\thanks{Corresponding author

E-mail addresses: nums8080@naver.com (M.Choi), srkim@snu.ac.kr (S.-R.Kim)
}
}

\affil[]{Department of Mathematics Education,
Seoul National University, Seoul 08826, Republic of Korea}
\begin{document}
\maketitle
\begin{abstract}
In this paper, we completely characterize the digraphs of order $n$ whose $m$-step competition graphs are star graphs for positive integers $2\leq m < n$.
This result in matrix version identifies the solution set to the matrix equation $X^m(X^T)^m= \Lambda_n+I_n$ for positive integers $2\leq m < n$ where $I_n$ is the identity matrix of order $n$ and $\Lambda_n$ is a $(0,1)$ Boolean matrix such that the first row and the first column consist of $1$'s except $(1,1)$-entry and the remaining entries are $0$, which is the adjacency matrix of a star graph of order $n$.

We also derive meaningful properties of the digraphs whose $m$-step competition graphs are trees. In the process,
we extend a result of Helleloid~[Connected triangle-free $m$-step competition graphs, Discrete Appl.\ Math.\ 145 (2005) 376--383] by showing that for all positive integers $m \geq 2$ and $n$, the connected triangle-free $m$-step competition graph on $n$ vertices is a tree.
\end{abstract}
\noindent
{\it Keywords.} $m$-step competition graph; triangle-free graph; tree-inducing digraph; tree; star graph.

\noindent
{{{\it 2010 Mathematics Subject Classification.} 05C20, 05C75}}
\section{Introduction}
For all undefined
graph theory terminologies, see \cite{bondy1976graph}.
In this paper, all the graphs and digraphs are assumed to be finite and
we only consider digraphs having no parallel arcs.

Let $D$ be a digraph and $m$ be a positive integer.
A vertex $y$ (resp.\ $x$) is an {\em $m$-step prey} (resp.\ {\em $m$-step predator}) of a vertex $x$ (resp.\ $y$) if and only if there exists a directed walk from $x$ to $y$ of length $m$.
The {\em $m$-step competition graph} of a digraph $D$, denoted by $C^m(D)$, has the same vertex set as $D$ and has an edge between two distinct vertices $u$ and $v$ if and only if there exists an $m$-step common prey of $u$ and $v$ in $D$. The notion of an $m$-step competition graph introduced by Cho~{\em et al.}~\cite{cho2000m} is a generalization of the competition graph introduced by Cohen~\cite{Cohen} (the {\em competition graph} of a digraph $D$ is $C^1(D)$).
Since its introduction, an $m$-step competition graph  has been extensively studied (see, for example, \cite{belmont2011complete,helleloid2005connected,ho2005m,park2011m,zhao2009note,eoh2020m,jung2022competition}). The author may refer \cite{factor20111,zhang20161,eoh2020niche,xiong2020competition} to for variants of competition graph other than $m$-step competition graphs.

We call a vertex of indegree $0$ and a vertex of outdegree $0$ in a digraph $D$ a {\em source} and a {\em sink}, respectively, of $D$. A sink of $D$ becomes an isolated vertex in $C^m(D)$.
Since we principally study the digraphs whose $m$-step competition graphs are connected for an integer $m \ge 2$, it is sufficient to consider digraphs without sinks.
In this context, throughout this paper, we assume that any digraph has no sink.

Given any digraph $D$, we can associate a graph $G$ on the same vertex set as $V(D)$
simply by replacing each arc $(u,v)$ with an edge $uv$. This graph is said to be the
{\it underlying graph} of $D$.
A digraph $D$ is called {\it weakly connected} if the underlying graph of $D$ is connected and a
{\it weak component} of $D$ is a subdigraph of $D$ induced by a component in the underlying graph of $D$.

For two vertex-disjoint weakly connected digraphs $D_1$ and $D_2$, it is true that
$C^m(D_1 \cup D_2)=C^m(D_1) \cup C^m(D_2)$ for any positive $m$.
In this vein, it is sufficient to consider weakly connected digraphs throughout this paper.
From now on, we assume that any digraph in this paper is weakly connected unless otherwise mentioned.

We call a complete bipartite graph $K_{1,l}$ for some positive integer $l$ a {\it star graph}.

In this paper, we show the following theorem (the definitions of a windmill digraph and an $m$-conveyor digraph will be given right after the theorem statement).
\begin{Thm} \label{thm:complete_star}
For positive integers $2\leq m < n$, the star graph is an $m$-step competition graph of a digraph $D$ with $n$ vertices if and only if one of the following holds:
\begin{itemize}
\item[(i)] $D$ is a windmill digraph;
\item[(ii)] $D$ is an $m$-conveyor digraph;
\item[(iii)] $m=2$ and $D$ is isomorphic to the digraph given in Figure~\ref{fig:loopless-stargraph}.
\end{itemize}
\end{Thm}
A {\it windmill digraph} is defined to be a digraph satisfying the following three conditions:
\begin{enumerate}
\item[($\text{W}1$)] $D$ has exactly one source $v$;
\item[($\text{W}2$)] $D-v$ is a vertex-disjoint union of directed cycles;
\item[($\text{W}3$)]
each vertex except $v$ is a prey of $v$
\end{enumerate}
(see the windmill digraphs of order $3$ in Figure~\ref{fig:example-star-generating-three} for an illustration).

\begin{figure}
\begin{center}
 \begin{tikzpicture}[auto,thick]
    \tikzstyle{player}=[minimum size=5pt,inner sep=0pt,outer sep=0pt,fill,color=black, circle]
    \tikzstyle{source}=[minimum size=5pt,inner sep=0pt,outer sep=0pt,ball color=black, circle]
    \tikzstyle{arc}=[minimum size=5pt,inner sep=1pt,outer sep=1pt, font=\footnotesize]
    \path (0:-1cm)   node [player]  (a) {};
    \path (0:0cm)     node [player]  (b)  {};
    \path (-90:1cm)      node [player]  (d) {};
   \draw[black,thick,-stealth] (a) - +(b);
   \draw[black,thick,-stealth] (a) - +(d);
   \draw[black,thick,-stealth] (b) to[in=100, distance=0.5cm](b);
   \draw[black,thick,-stealth] (d) to[in=100, distance=0.5cm](d);
   \path (-100:1.8cm)      node  (d) {$D_1$};
    \end{tikzpicture}
    \qquad  \qquad
    \begin{tikzpicture}[auto, thick]
    \tikzstyle{player}=[minimum size=5pt,inner sep=0pt,outer sep=0pt,fill,color=black, circle]
    \tikzstyle{source}=[minimum size=5pt,inner sep=0pt,outer sep=0pt,ball color=black, circle]
    \tikzstyle{arc}=[minimum size=5pt,inner sep=1pt,outer sep=1pt, font=\footnotesize]
    \path (0:-1cm)   node [player]  (a) {};
    \path (0:0cm)     node [player]  (b)  {};
    \path (-90:1cm)      node [player]  (d) {};
    \path (-100:1.8cm)      node  (e) {$D_2$};
   \draw[black,thick,-stealth] (a) - +(b);
   \draw[black,thick,-stealth] (a) - +(d);
   \draw[black,thick,-stealth] (b) to[in=90, out=270](d);
   \draw[black,thick,-stealth] (d) to[in=0, out=0](b);
    \end{tikzpicture}
 \end{center}
 \caption{The windmill digraphs with three vertices.}
\label{fig:example-star-generating-three}
\end{figure}
We call a nontrivial directed path or cycle connecting vertices of indegree $2$ an {\em internally secure lane} if each of its interior vertices has indegree $1$.

We call a digraph $D$ an {\it $m$-conveyor digraph} for some $m \geq 2$ if
$D$ has a vertex $v$ satisfying the following conditions:
\begin{enumerate}
\item[(M1)] $v$ is the only predator of $v$;
\item[(M2)] $D-v$ is a vertex-disjoint union of directed cycles;
\item[(M3)] each internally secure lane in $D$ has length at most $m$
\end{enumerate}
(see the $2$-conveyor digraphs of order $4$ in Figure~\ref{fig:example-star-generating-typeB} for an illustration).
\begin{figure}
\begin{center}
 \begin{tikzpicture}[auto,thick]
    \tikzstyle{player}=[minimum size=5pt,inner sep=0pt,outer sep=0pt,fill,color=black, circle]
    \tikzstyle{source}=[minimum size=5pt,inner sep=0pt,outer sep=0pt,ball color=black, circle]
    \tikzstyle{arc}=[minimum size=5pt,inner sep=1pt,outer sep=1pt, font=\footnotesize]
    \path (0:-1cm)   node [player]  (a) [label=left:$v$] {};
    \path (0:0cm)     node [player]  (b)  {};
    \path (90:1cm)   node [player]  (c)  {};
    \path (-90:1cm)      node [player]  (d) {};
   \draw[black,thick,-stealth] (a) - +(b);
   \draw[black,thick,-stealth] (a) - +(c);
   \draw[black,thick,-stealth] (a) - +(d);
   \draw[black,thick,-stealth] (a) to[in=100, distance=0.5cm](a);
   \draw[black,thick,-stealth] (b) to[in=100, distance=0.5cm](b);
   \draw[black,thick,-stealth] (c) to[in=100, distance=0.5cm](c);
   \draw[black,thick,-stealth] (d) to[in=100, distance=0.5cm](d);
    \end{tikzpicture}
     \qquad  \qquad
    \begin{tikzpicture}[auto, thick]
    \tikzstyle{player}=[minimum size=5pt,inner sep=0pt,outer sep=0pt,fill,color=black, circle]
    \tikzstyle{source}=[minimum size=5pt,inner sep=0pt,outer sep=0pt,ball color=black, circle]
    \tikzstyle{arc}=[minimum size=5pt,inner sep=1pt,outer sep=1pt, font=\footnotesize]
    \path (0:-1cm)   node [player]  (a) [label=left:$v$] {};
    \path (0:0cm)     node [player]  (b)  {};
    \path (90:1cm)   node [player]  (c)  {};
    \path (-90:1cm)      node [player]  (d) {};
   \draw[black,thick,-stealth] (a) - +(b);
   \draw[black,thick,-stealth] (a) - +(c);
   \draw[black,thick,-stealth] (a) to[in=100, distance=0.5cm](a);
   \draw[black,thick,-stealth] (c) to[in=100, distance=0.5cm](c);
   \draw[black,thick,-stealth] (b) to[in=90, out=270](d);
   \draw[black,thick,-stealth] (d) to[in=0, out=0](b);
    \end{tikzpicture}
     \qquad  \qquad
    \begin{tikzpicture}[auto, thick]
    \tikzstyle{player}=[minimum size=5pt,inner sep=0pt,outer sep=0pt,fill,color=black, circle]
    \tikzstyle{source}=[minimum size=5pt,inner sep=0pt,outer sep=0pt,ball color=black, circle]
    \tikzstyle{arc}=[minimum size=5pt,inner sep=1pt,outer sep=1pt, font=\footnotesize]
    \path (0:-1cm)   node [player]  (a) [label=left:$v$] {};
    \path (0:0cm)     node [player]  (b)  {};
    \path (90:1cm)   node [player]  (c)  {};
    \path (-90:1cm)      node [player]  (d) {};
   \draw[black,thick,-stealth] (a) - +(b);
   \draw[black,thick,-stealth] (a) - +(c);
   \draw[black,thick,-stealth] (a) - +(d);
   \draw[black,thick,-stealth] (a) to[in=100, distance=0.5cm](a);
   \draw[black,thick,-stealth] (c) to[in=100, distance=0.5cm](c);
   \draw[black,thick,-stealth] (b) to[in=90, out=270](d);
   \draw[black,thick,-stealth] (d) to[in=0, out=0](b);
    \end{tikzpicture}

    \begin{tikzpicture}[auto,thick]
    \tikzstyle{player}=[minimum size=5pt,inner sep=0pt,outer sep=0pt,fill,color=black, circle]
    \tikzstyle{source}=[minimum size=5pt,inner sep=0pt,outer sep=0pt,ball color=black, circle]
    \tikzstyle{arc}=[minimum size=5pt,inner sep=1pt,outer sep=1pt, font=\footnotesize]
    \path (0:-1cm)   node [player]  (a) [label=left:$v$] {};
    \path (0:0cm)     node [player]  (b)  {};
    \path (90:1cm)   node [player]  (c)  {};
    \path (-90:1cm)      node [player]  (d) {};
   \draw[black,thick,-stealth] (a) - +(b);
   \draw[black,thick,-stealth] (a) - +(c);
   \draw[black,thick,-stealth] (a) to[in=100, distance=0.5cm](a);
   \draw[black,thick,-stealth] (c) to[in=90, out=270](b);
   \draw[black,thick,-stealth] (b) to[in=90, out=270](d);
   \draw[black,thick,-stealth] (d) to[in=10, out=0](c);
    \end{tikzpicture}
             \qquad  \qquad
 \begin{tikzpicture}[auto,thick]
    \tikzstyle{player}=[minimum size=5pt,inner sep=0pt,outer sep=0pt,fill,color=black, circle]
    \tikzstyle{source}=[minimum size=5pt,inner sep=0pt,outer sep=0pt,ball color=black, circle]
    \tikzstyle{arc}=[minimum size=5pt,inner sep=1pt,outer sep=1pt, font=\footnotesize]
    \path (0:-1cm)   node [player]  (a) [label=left:$v$] {};
    \path (0:0cm)     node [player]  (b)  {};
    \path (90:1cm)   node [player]  (c)  {};
    \path (-90:1cm)      node [player]  (d) {};
   \draw[black,thick,-stealth] (a) - +(b);
   \draw[black,thick,-stealth] (a) - +(c);
   \draw[black,thick,-stealth] (a) - +(d);
   \draw[black,thick,-stealth] (a) to[in=100, distance=0.5cm](a);
   \draw[black,thick,-stealth] (c) to[in=90, out=270](b);
   \draw[black,thick,-stealth] (b) to[in=90, out=270](d);
   \draw[black,thick,-stealth] (d) to[in=10, out=0](c);
    \end{tikzpicture}
 \end{center}
 \caption{The $2$-conveyor digraphs of order $4$}
\label{fig:example-star-generating-typeB}
\end{figure}

\begin{figure}
\begin{center}
    \begin{tikzpicture}[auto,thick]
    \tikzstyle{player}=[minimum size=5pt,inner sep=0pt,outer sep=0pt,fill,color=black, circle]
    \tikzstyle{source}=[minimum size=5pt,inner sep=0pt,outer sep=0pt,ball color=black, circle]
    \tikzstyle{arc}=[minimum size=5pt,inner sep=1pt,outer sep=1pt, font=\footnotesize]
    \path (0:0cm)   node [player]  (u)  [label=below:$u$] {};
    \path (45:1.5cm)     node [player]  (w)  [label=above:$y$] {};
    \path (0:1.5cm)     node [player]  (v)  [label=below:$z$] {};
   \draw[black,thick,-stealth] (w) to [in=45, out=225, distance=0.5cm](u);
      \draw[black,thick,-stealth] (u) to [in=180, out=90, distance=0.5cm](w);
  \draw[black,thick,-stealth] (v) to [out=110, in=-65, distance=0.5cm](w);
   \draw[black,thick,-stealth] (u) to[out=0, in=180,  distance=0.5cm](v);
    \end{tikzpicture}
\quad \quad \quad \quad
    \begin{tikzpicture}[auto,thick]
    \tikzstyle{player}=[minimum size=5pt,inner sep=0pt,outer sep=0pt,fill,color=black, circle]
    \tikzstyle{source}=[minimum size=5pt,inner sep=0pt,outer sep=0pt,ball color=black, circle]
    \tikzstyle{arc}=[minimum size=5pt,inner sep=1pt,outer sep=1pt, font=\footnotesize]
    \path (0:0cm)   node [player]  (u)  [label=below:$u$] {};
    \path (45:1.5cm)     node [player]  (w)  [label=above:$y$] {};
    \path (0:1.5cm)     node [player]  (v)  [label=below:$z$] {};
  \draw[black,thick,-] (u) - +(w);
  \draw[black,thick,-] (u) - +(v);
    \end{tikzpicture}
\caption{A digraph and its $2$-step competition graph}
\label{fig:loopless-stargraph}
\end{center}
\end{figure}
The adjacency matrix of a windmill digraph 
is in the form of the first matrix given in Figure~\ref{fig:matrix_windmill}.
Here, $\Gamma_n$ is the adjacency matrix of a directed cycle of length $n$, that is, 
\[(\Gamma_n)_{ij}=\begin{cases} 1 & \text{if $j=i+1$ or $(i,j)=(n,1)$}; \\
0 & \text{otherwise.}\end{cases}\]
The adjacency matrix of an $m$-conveyor digraph is in the form of the second matrix given in Figure~\ref{fig:matrix_windmill}.
The first row represents $v$ satisfying (M1) and (M2).
By (M3), the $(0,1)$ nonzero matrix $Q^{(m)}_k$ has  size $1\times k$ and satisfies the following properties:
\begin{enumerate}
\item[(F1)]  the number of consecutive zeros is at most $m-1$;
\item[(F2)] if 
the $(1,1)$-entry and the $(1,k)$-entry equal $0$, then the number of first consecutive zeros and that of last consecutive zeros add up to at most $m-1$.
\end{enumerate}

\begin{figure}
\begin{equation*}
\begin{bmatrix}[cccc]
0   &  J & \cdots & J    \\
                   O    & \Gamma_{n_1}  & O & O   \\
  \vdots   & O  & \ddots & O   \\
                     O  &  O & O & \Gamma_{n_k}   \\
\end{bmatrix}
, \quad
\begin{bmatrix}[cccc]
 1  &  Q^{(m)}_{n_1} & \cdots & Q^{(m)}_{n_k}   \\
                  O    & \Gamma_{n_1}  & O & O   \\
  \vdots   & O  & \ddots & O   \\
                     O  &  O & O & \Gamma_{n_k}   \\
\end{bmatrix}
, \quad
\begin{bmatrix}[ccc]
0   &  1 & 1    \\
1 & 0  & 0  \\
0 & 1 & 0   \\
\end{bmatrix}
\end{equation*}
\caption{
Adjacency matrices of a windmill digraph, a $m$-conveyor digraph, and the digraph given in Figure \ref{fig:loopless-stargraph}, respectively, where the blocks $J$ and $O$ stand for a matrix of all $1$'s and a zero matrix, respectively.
}
\label{fig:matrix_windmill}
\end{figure}

Theorem~\ref{thm:complete_star} may be restated in terms of matrices.
For the two-element Boolean algebra $\mathcal{B}=\{0,1\}$, $\mathcal{B}_n$ denotes the set of all $n \times n$ matrices over $\mathcal{B}$.
Under the Boolean operations ($1 + 1 = 1$, $0 + 0 = 0$, $1 + 0 = 1$, $1 \times 1 = 1$, $0 \times 0 = 0$, $1 \times 0 = 0$), matrix addition and multiplication are still well-defined in $\mathcal{B}_n$.
Throughout this paper, a matrix is Boolean unless otherwise mentioned.

We note that the adjacency matrix of $C^m(D)$ for a digraph $D$ of order $n$ is the matrix $A^*_{m}$ obtained from $A^m(A^T)^m$ by replacing each of diagonal element with $0$ where $A$ is the adjacency matrix of $D$.
To see why, we take two distinct vertices $u$ and $v$ of $D$ and suppose that the $i$th row and the $j$th row are the rows corresponding to $u$ and $v$, respectively.
Then
\begin{tabbing}
\ \ \ \ \ \ \= $u$ and $v$ are adjacent in $C^m(D)$ \\
$\Leftrightarrow$ \>  $u$ and $v$ have an $m$-step common prey in $D$ \\
$\Leftrightarrow$ \> inner product of the $i$th row and the $j$th row of $A^m$ is $1$\\
$\Leftrightarrow$ \>  the $(i,j)$-entry of $A^*_m$ is $1$.
\end{tabbing}
Thus
$u$ and $v$ are adjacent in $C^m(D)$ if and only if the $(i,j)$-entry of $A^*_m$ is $1$.
Therefore we have the following corollary restating Theorem~\ref{thm:complete_star} in terms of matrices:
\begin{Cor}[Matrix version]
 For positive integers $2\leq m < n$, a square matrix $X$ of order $n$ satisfies $X^m(X^T)^m=\Lambda_n+I_n$ if and only if
$P^TXP$ for some permutation matrix $P$ of order $n$ is one of the matrices given in Figure~\ref{fig:matrix_windmill}, where $I_n$ is the identity matrix of order $n$ and $\Lambda_n$ is the square matrix of order $n$ with the first row and first column of $\Lambda_n$
consisting of $1$'s except $(1,1)$-entry and the remaining entries being $0$.
\end{Cor}

We also prove the following result.
\begin{Thm} \label{thm:triangle-free-tree}
For all positive integers $2\leq m<n$
the connected triangle-free $m$-step competition graph on $n$ vertices is a tree.
\end{Thm}
Even for a digraph $D$ and an integer $m > |V(D)|$, the same is true as follows.
\begin{Thm}
[Helleloid \cite{helleloid2005connected}]
\label{thm:traingle-free}
 For all positive integers $m \geq n$, the only connected triangle-free $m$-step competition graph on $n$ vertices is the star graph.
\end{Thm}

In 2000, Cho \emph{et al.}\cite{cho2000m} posed the following question: For which values of $m$ and $n$ is $P_n$ an $m$-step competition graph?
In 2005, Helleloid \cite{helleloid2005connected} partially answered the question and study connected triangle-free $m$-step competition graphs and obtain Theorem~\ref{thm:traingle-free}.
By Theorems~\ref{thm:triangle-free-tree} and~\ref{thm:traingle-free}, we have the following more general result.
\begin{Cor}\label{thm:expansion}
For all positive integers $m\geq 2$ and $n$,
the connected triangle-free $m$-step competition graph on $n$ vertices is a tree.
\end{Cor}
As the rest of this paper is devoted to proving Theorems~\ref{thm:complete_star} and \ref{thm:triangle-free-tree},  we may assume from now on that $m\geq 2$ and $m < n$ whenever we are given a digraph of order $n$ whose $m$-step competition graph is triangle-free.

\section{The triangle-free $m$-step competition graphs}
In this section, we show that all the connected triangle-free $m$-step competition graphs are trees.

Let $D$ be a digraph and $v$ be a vertex of $D$.
We denote the $m$-step prey of $v$ by $N^+_{D^m}(v)$ and 
the $m$-step predators of $v$ by $N^-_{D^m}(v)$, respectively.
When no confusion is likely, we will just write $N^+_m(v)$ and $N^-_m(v)$.
We note that $N^+_1(v)=N^+(v)$ and $N^-_1(v)=N^-(v)$.
Technically, we write $N^+_0(v)=N^-_0(v)=\{v\}$.
We call $|N^-_i(v)|$ and $|N^+_i(v)|$ the {\it$i$-step indegree} and the {\it$i$-step outdegree} of $v$, respectively, and denote them by $d^-_i(v)$ and $d^+_i(v)$, respectively.
We note that $d^+_1(v)=d^+(v)$ and $d^-_1(v)=d^-(v)$.

We make the following useful observations.

\begin{Lem}\label{lem:triangle-free-chars}
Let $D$ be a digraph such that $C^m(D)$ is triangle-free. Then the following are true:
\end{Lem}
\begin{enumerate}[{(1)}]
\item Any vertex in $D$ has $i$-step outdegree at least $1$ for any positive integer $i$.
\item Any vertex in $D$ has $i$-step indegree at most $2$ for any positive integer $i \leq m$.
\item 
If a directed walk contains at least two vertices and its origin and terminus have indegree $2$, then it is a juxtaposition of internally secure lanes. 
\item
For any two internally secure lanes $W$ and $W'$ in $D$ starting at $w$ and $w'$, respectively, and
sharing $v$ as an interior vertex,
the $(w,v)$-section of $W$ and the $(w',v)$-section of $W'$ coincide.
\end{enumerate}
\begin{proof}
	Since we have assumed that any digraph has no sinks, part (1) is true.

To prove part (2), suppose, to the contrary,
$d^-_i(u) \geq 3$ for some vertex $u$ of $D$ and a positive integer $i \leq m$. Then there exist three distinct $i$-step predators $x$, $y$, and $z$ of $u$. By part (1), $u$ has an $(m-i)$-step prey $v$. Then $v$ is an $m$-step common prey of $x,y$, and $z$. Thus $x,y$, and $z$ form a triangle in $C^m(D)$, a contradiction.
Hence part (2) is true.

Part (3)  immediately follows from the definition of internally secure lane and part (2).

To show part (4),
let $W=w \to v_{1} \to \cdots \to v_l$ and $W'=w'\to v'_1 \to \cdots \to v'_{l'} $ be a pair of internally secure lanes sharing $v$ as an interior vertex for some positive integers $l$ and $l'$.
Then $v=v_k=v'_{k'}$ for some $k \in \{1,\ldots,l-1\}$ and $k' \in \{1,\ldots,l'-1\}$.
By the definition of internally secure lane,
$d^-(v_i)=d^-(v'_{i'})=1$ for each $1\leq i\leq k$ and $1\leq i'\leq k'$.
Therefore the $(w,v_k)$-section of $W$ and the $(w',v'_{k'})$-section of $W'$ coincide.
Thus part (4) is true.
\end{proof}

 \begin{Thm} \label{thm:fundamental}
Let $G$ be the $m$-step competition graph of a digraph $D$ such that $G$ is triangle-free and has the edges as many as the vertices.
Then the following are true:
\begin{itemize}
	\item[(1)] For each vertex $u$ of outdegree at least $2$ in $D$, each prey of $u$ has indegree $2$ in $D$.
	\item[(2)] Every vertex in $D$ lies on some internally secure lane.
	\item[(3)] Each internally secure lane of $D$ has length $m$.
\end{itemize}
	\end{Thm}
\begin{proof}
We consider a set \[A=\{(u,\{v,w\})  \mid v\neq w, \{v,w\} \subseteq N^-_m(u)\}. \]
By the definition of $m$-step competition graph, $|A| \geq |E(G)|$.
Thus, by the definition of $A$ and Lemma~\ref{lem:triangle-free-chars}(2),
\[ |E(G)|\leq |A| = \sum_{v \in V(D)} {d^-_m(u) \choose 2 } \leq \sum_{v \in V(D)} {2\choose 2 } = |V(D)|=|V(G)|.\]
Then, since $|E(G)|=|V(G)|$ by the hypothesis,
\begin{equation} \label{eq:connect-triang-tree-1}
 d^-_m(v)=2
\end{equation} for each vertex $v$ in $D$.
In addition, if $u$ and $v$ are adjacent in $G$, then
\[
|N^+_m(u) \cap N^+_m(v) | =1,
\]
so, for each pair of vertices $u$ and $v$ in $D$,
\[
|N^+_m(u) \cap N^+_m(v) | \leq 1
\]
and
\begin{equation} \label{eq:-}
 |N^-_m(u) \cap N^-_m(v) | \leq 1.
 \end{equation}
Suppose for a contradiction that there exist two vertices $u$ and $v$ such that $|N^+_j(u) \cap N^+_j(v) | \geq 2$ for some positive
integer $j < m$.
Take two distinct vertices $w_1$ and $w_2$ in $N^+_j(u) \cap N^+_j(v)$.
Then $\{u,v\} \subseteq N^-_j(w_1)\cap N^-_j(w_2)$.
Therefore $N^-_j(w_1)= N^-_j(w_2)=\{u,v\}$ by Lemma~\ref{lem:triangle-free-chars}(2).
Thus $N^-_m(w_1)= N^-_m(w_2)$.
Then, since $d^-_m(w_1)=d^-_m(w_2)=2$ by \eqref{eq:connect-triang-tree-1},
$|N^-_m(w_1) \cap N^-_m(w_2)|=2$, which contradicts \eqref{eq:-}.
Therefore, for each pair of vertices $u$ and $v$,
\begin{equation} \label{eq:connect-triang-tree-4}
|N^+_i(u) \cap N^+_i(v) | \leq 1
\end{equation} for any positive
integer $i \leq m$.

To show part (1) by contradiction, suppose that there exist a vertex $u$ of outdegree at least $2$ and a prey $v$ of $u$ has indegree not equal to $2$ in $D$.
Then, by Lemma~\ref{lem:triangle-free-chars}(2),
$v$ has indegree $1$.
In addition, since $u$ has outdegree at least $2$, $u$ has a prey $w$ other than $v$.
By \eqref{eq:connect-triang-tree-1},
 $N^-_m(v)=\{x,y\}$ for some vertices $x$ and $y$ in $D$.
 Since $u$ is the only predator of $v$,
$N^-_{m-1}(u)=\{x,y\}$.
Therefore $ \{x,y\}\subseteq N^-_m(w)$ and so $\{x,y\} \subseteq N^-_m(v) \cap N^-_m(w)$, which contradicts \eqref{eq:-}.
Hence part (1) is true.

To show part (2), take a vertex $v$ in $D$.
If any $i$-step prey of $v$ has indegree at most $1$ for each $1\leq i \leq m$,
then there is an $m$-step prey of $v$ having $m$-step indegree $1$, which contradicts~\eqref{eq:connect-triang-tree-1}.
Therefore there exists a $j$-step prey $x$ of $v$ having indegree at least $2$ for some $j\in \{1,\ldots,m\}$.
Thus $x$ has indegree $2$ by Lemma~\ref{lem:triangle-free-chars}(2).
If $v$ has indegree $2$,
then a directed $(v,x)$-walk contains an internally secure lane on which $v$ lies.
Suppose that $v$ has indegree not equal to $2$.
Then $v$ has indegree $1$ by Lemma~\ref{lem:triangle-free-chars}(2) and \eqref{eq:connect-triang-tree-1}.
If each $i$-step predator of $v$ has indegree at most $1$ for each $1\leq i \leq m-1$,
then $d^-_m(v)<2$, a contradiction to \eqref{eq:connect-triang-tree-1}.
Thus there exists a $k$-step predator $y$ of $v$ having indegree at least $2$ for some $k\in \{1,\ldots,m-1\}$.
Then, by Lemma~\ref{lem:triangle-free-chars}(2),
each of $x$ and $y$ has indegree $2$.
Hence we may conclude that every vertex in $D$ lies on some internally secure lane.

To show part (3), take an internally secure lane $W:=v_{0} \to v_{1} \to \cdots \to v_j$ for some positive integer $j$.
Then, by the definition of an internally secure lane,
\[
N^-(v_0)=\{x,y\}, \quad N^-(v_j) = \{v_{j-1},w\}
\]
for some vertices $x,y$, and $w$ in $D$.
Then
\begin{equation}\label{eq:connect-triang-tree-2-1-1-0}
N^-_{i+1}(v_i) \supseteq \{x,y\}
\end{equation}
for each $0\leq i \leq j$.
If $j \geq 2$, then
\begin{equation}\label{eq:connect-triang-tree-2-1-1}
N^-(v_{i+1})=\{v_{i}\}
\end{equation}
for each $0 \leq i \leq j-2$ by the definition of internally secure lane  and so, by part (1),
\begin{equation}\label{eq:connect-triang-tree-3-1}
N^+(v_i)=\{v_{i+1}\}
\end{equation}
 for each $0 \leq i \leq j-2$.

To reach a contradiction, suppose $j \neq m$.
If $j > m$, then $j \geq 2$ and so, by~\eqref{eq:connect-triang-tree-2-1-1},
$N^-_m(v_{j-1})=\{v_{j-m-1}\}$, which contradicts \eqref{eq:connect-triang-tree-1}.
Therefore $j <m$.
Then
\[N^-_{i+1}(v_i) = \{x,y\}\] for each $0 \le i \le j$ by~\eqref{eq:connect-triang-tree-2-1-1-0} and Lemma~\ref{lem:triangle-free-chars}(2).
Since a $j$-step predator of $w$ is a $(j+1)$-step predator of $v_j$, $N^-_{j}(w) \subseteq \{x,y\}$.
Since $N^-_{j}(v_{j-1})= \{x,y\}$,
$N^-_{j}(w) \subsetneq \{x,y\}$ by~\eqref{eq:connect-triang-tree-4}.
Then, since $D$ has no source by~\eqref{eq:connect-triang-tree-1},
$N^-_{j}(w) = \{x\}$ or $\{y\}$.
Without loss of generality, we may assume $N^-_{j}(w) = \{x\}$.
Then there exists a directed walk $W^*$ of length $j$ from $x$ to $w$.
If $j=1$, then $d^-(w)=1$, $w$ is a prey of $x$, and $\{v_0,w\} \subseteq N^+(x)$, which contradicts part (1).
Therefore $j\geq 2$.
Let $z$ be the vertex which $x$ is immediately going toward on $W^*$.
If $z=v_0$, then
$w=v_{j-1}$ by~\eqref{eq:connect-triang-tree-3-1} and we reach a contradiction.
Therefore $z \neq v_0$.
Thus $x$ has outdegree at least $2$.
Hence $z$ has a predator $x'$ other than $x$ by part (1).
Then attaching the arc $(x',z)$ to the $(z,w)$-section of $W^*$ results in a directed walk of length $j$ from $x'$ to $w$.
Therefore $\{x,x'\} \subseteq N^-_{j}(w)$, which contradicts the assumption that $N^-_{j}(w) = \{x\}$.
Thus we may conclude that $j=m$.
Since $W$ was arbitrarily chosen from $D$, part (3) is valid.
\end{proof}

Given an internally secure lane $W=v_0 \to v_1 \to \cdots \to v_m$, we call $v_k$ the {\it $k$th interior vertex} of $W$ for each $1\leq k <m$.

\begin{Thm} \label{thm:triangle-free-m-components}
If the $m$-step competition graph of a digraph is triangle-free and has the edges as many as the vertices, then it has at least $m$ components.
\end{Thm}
\begin{proof}
Suppose that there exists a digraph $D$ whose $m$-step competition graph $G$ is triangle-free and has the edges as many as the vertices.
Take an internally secure lane $W$ in $D$ (it exists since every vertex in $D$ lies on some internally secure lane by Theorem~\ref{thm:fundamental}(2)).
By Theorem~\ref{thm:fundamental}(3), $W$ has length $m$. Let $W=v_{0} \to v_{1} \to \cdots \to v_m$ and take $v_k$ for some $k \in \{1,\ldots, m-1\}$.
Suppose that
 there exists a vertex $w_k$ adjacent to $v_k$ in $G$.
 Then they have an $m$-step common prey
 in $D$ and so $v_k$ and $w_k$ have an $l$-step common prey $z$ that has indegree at least $2$ in $D$ for some $l \in \{1,\ldots,m\}$.
 Therefore
 $z$ has indegree $2$ by Lemma~\ref{lem:triangle-free-chars}(2).
Then there exist a directed $(v_k,z)$-walk $W_1$ of length $l$ and a directed $(w_k,z)$-walk $W_2$ of length $l$.
Since $z$ has indegree $2$,
the directed walk $W'$ obtained by concatenating the $(v_0,v_k)$-section of $W$ and $W_1$ contains an internally secure lane.
We note that the origin and the terminus of $W'$ have indegree $2$.
Then, since the length of $W'$ is at most $2m-1$, $W'$ must be an internally secure lane by Lemma~\ref{lem:triangle-free-chars}(3) and Theorem~\ref{thm:fundamental}(3). Thus $l=m-k$.
By Theorem~\ref{thm:fundamental}(3) again, $W_2$ must be a section of an internally secure lane of length $l$.
Thus we may conclude that
\begin{enumerate}
\item[($\star$)]
each vertex  adjacent to the  $k$th interior vertex of an internally secure lane is the $k$th interior vertex of an internally secure lane.
\end{enumerate}
 Now, for each $1\leq i \leq m-1$,
 we define a vertex set $\mathcal{V}_i$ as follows:
  $v\in \mathcal{V}_i$ if and only if
 $v$ is the $i$th interior vertex of some internally secure lane.
 Then $\mathcal{V}_i \cap \mathcal{V}_j = \emptyset$
 if $i \neq j$ by Lemma~\ref{lem:triangle-free-chars}(4).
Moreover, since $v_i \in \mathcal{V}_i $,
$ \mathcal{V}_i \neq \emptyset $ for each $1\leq i \leq m-1$.

Now we choose $j \in \{1,\ldots,m-1\}$.
Then take a vertex $v$ in $\mathcal{V}_j$ and
let $X$ be the component containing $v$ in $G$.
Take a vertex $w$ in $X$.
Then there exists a shortest path $P$ from $v$ to $w$.
By repeatedly applying ($\star$) to each vertex on $P$ from the nearest to the farthest from $v$, we may show that $w \in \mathcal{V}_j$.
Therefore $X \subseteq \mathcal{V}_j$.
Then, since $\mathcal{V}_i \cap \mathcal{V}_j = \emptyset$ for $i \neq j$ and $\mathcal{V}_k \neq \emptyset$ for each $1\leq k \leq m-1$,
 $G$ has at least $m-1$ components each of which is included in $\mathcal{V}_j$ for some $j \in \{1,\ldots,m-1\}$.
We note that each vertex in $\bigcup_{i=1}^{m-1}\mathcal{V}_i$ is an interior vertex of an internally secure lane and so, by the definition of internally secure lane, it has indegree $1$ for each $1\leq i \leq m-1$.
Therefore the origin and the terminus of any internally secure lane in $D$ cannot belong to any of the components obtained previously.
Hence $G$ has at least $m$ components.
\end{proof}
 Now we are ready to prove Theorem~\ref{thm:triangle-free-tree}.

\begin{proof}[Proof of Theorem~\ref{thm:triangle-free-tree}]
Suppose, to the contrary, that there exists a digraph $D$ such that $C^m(D)$ is triangle-free and connected but is not a tree.
Then, since $C^m(D)$ is connected but is not a tree, $|E(C^m(D))| > |V(C^m(D))|-1$.
By Lemma~\ref{lem:triangle-free-chars}(2), $d^-_m(v)=2$ for each vertex $v$ in $D$. Therefore $|E(C^m(D))| \leq |V(C^m(D))|$ and so $|E(C^m(D))| =|V(C^m(D))|$.
Thus $C^m(D)$ is disconnected by Theorem~\ref{thm:triangle-free-m-components} and we reach a contradiction.
Hence $C^m(D)$ is a tree.
\end{proof}

 \section{Digraphs whose $m$-step competition graphs are trees}\label{sec:digraphs_tree}
In this section, we deduce basic properties of digraphs whose $m$-step competition graphs are trees.

We call a digraph $D$ with at least three vertices an {\it $m$-step tree-inducing digraph } if the $m$-step competition graph of $D$ is a tree for some integer $m \geq 2$. A digraph is said to be a {\it tree-inducing digraph } if it is an $m$-step tree-inducing digraph
for some integer $m \geq 2$.

\begin{Prop} \label{prop:proposition2-5}
 Let $D$ be an $m$-step tree-inducing digraph. Then $N^+_i(u)\neq N^+_i(v)$ for any distinct vertices $u$ and $v$ in $D$ and any  positive integer $i \leq m$.
\end{Prop}

\begin{proof}
Suppose $N^+_i(u)=N^+_i(v)$ for some distinct $u$ and $v$ in $D$ and a positive integer $i \leq m$.
Then $N^+_m(u)=N^+_m(v)$.
Since $N^+_m(u)\neq\emptyset$ by Lemma~\ref{lem:triangle-free-chars}(1),
$u$ and $v$ are adjacent in $C^m(D)$.
Moreover, $N^-_m(w)=\{u,v\}$
for each vertex $w \in N^+_m(u)$
by Lemma~\ref{lem:triangle-free-chars}(2).
Therefore an edge $uv$ is a component in $C^m(D)$, a contradiction to the connectedness of $C^m(D)$.
\end{proof}

\begin{Prop}
 [Helleloid \cite{helleloid2005connected}]\label{Prop:Helle}
 Let $D$ be a digraph with $n$ vertices whose $m$-step competition graph $C^m(D)$ is a tree.
   Then there is a one-to-one correspondence between the $n-1$ pairs of adjacent vertices in $C^m(D)$ and $n-1$ of $n$ vertices of D; namely, all but one vertex in $D$ serves as the $m$-step common prey for exactly one pair of adjacent
vertices in $C^m(D)$. The remaining vertex of $D$ can either be the m-step prey of no vertices, of any one vertex, or of any two
vertices adjacent in $C^m(D)$.
\end{Prop}
Based upon the above proposition, Belmont \cite{belmont2011complete} called the remaining vertex $\alpha$ in $D$ not assigned in a bijection between the edges of $C^m(D)$ and $n-1$ of the $n$ vertices of $D$ {\it anomaly}.
The author observed that
if the remaining vertex of $D$ is the $m$-step prey of any two vertices adjacent in $C^m(D)$, then
 the anomaly is not well-defined since there are two vertices with this property and went on to call the one arbitrarily chosen between the two vertices an anomaly.
By the definition of an anomaly,
it is clear that each tree-inducing digraph has a unique anomaly.

The following proposition gives a necessary and sufficient condition for a vertex of a tree-inducing digraph $D$ being the anomaly, which is actually a restatement of Proposition~\ref{Prop:Helle}.

\begin{Prop} \label{prop:idle-if-only-if}
Let $D$ be an $m$-step tree-inducing digraph.
Then $\alpha$ in $D$ is the anomaly if and only if $\alpha$ has either at most one $m$-step predator in $D$ or exactly two $m$-step predators that have another vertex $\beta$ as an $m$-step common prey in $D$.
Furthermore, if the latter of the ``if" part is true, then $\alpha$  and $\beta$ are the only vertices that share two $m$-step common predators.
\end{Prop}

\begin{Cor} \label{cor:anomaly-finder}
Let $D$ be an $m$-step tree-inducing digraph.
Then the following are true:
\begin{enumerate}[{(1)}]
\item
If $|N^+_m(u) \cap N^+_m(v)| \geq 2$ for some $u$ and $v$ in $D$, then the anomaly is contained in $N^+_m(u) \cap N^+_m(v)$.
\item If $d^-_m(v) \leq 1$, then $v$ is the anomaly.
\end{enumerate}
\end{Cor}
\begin{Cor} \label{cor:no-more-two-common}
Let $D$ be an $m$-step tree-inducing digraph.
For the anomaly $\alpha$ of $D$, exactly one of the following is true:
\begin{enumerate}[{(i)}]
\item $\alpha$ has
exactly two $m$-step predators that have  a vertex $v$ other than $\alpha$ as an $m$-step common prey in $D$, and $\alpha$ and $v$ are the only vertices that share two $m$-step common predators;
\item $\alpha$ has at most
one $m$-step predator and each vertex except $\alpha$ has exactly two $m$-step predators.
\end{enumerate}
\end{Cor}

\begin{Thm} \label{thm:triangle-free-chars}
Let $D$ be a digraph such that $C^m(D)$ is triangle-free and connected.
Then the following are true:
\begin{enumerate}[{(1)}]
\item $|\bigcup_{v\in U}N^+(v)| \geq |U| $ for any proper subset $U$ of $V(D)$.

\item For any vertices $u $ and $v$ in $D$, 
$|N^+_i(u) \cap N^+_i(v) | \leq |N^+_j(u) \cap N^+_j(v) | $ for any positive integers $i,j$ with $i \leq j \leq m$.

\item For each vertex $v$ in $D$, 
$d^+_i(v) \leq d^+_j(v) $  for any positive integers $i,j$ with $i\leq j \leq m$.

\end{enumerate}
\end{Thm}
\begin{proof}
We begin with the proof of the following claim:
\begin{Claim*} \label{claim:first}
For any nonempty proper subset $U$ of $V(D)$, there exists a vertex $u \in \bigcup_{v\in U}N^+(v)$ such that $|N^-(u) \cap U|= 1$.
\end{Claim*}
To reach a contradiction, suppose that there exists a nonempty proper subset $U^*$ of $V(D)$ 
with $|N^-(v)\cap U^*| \neq 1$ for each vertex $v$ in $\bigcup_{v\in U^*}N^+(v)$. Since any vertex in $\bigcup_{v\in U^*}N^+(v)$ is a prey of a vertex in $U^*$, $|N^-(v)\cap U^*| \geq 1$ for each vertex $v$ in $\bigcup_{v\in U^*}N^+(v)$ and so, by Lemma~\ref{lem:triangle-free-chars}(2), $|N^-(v)\cap U^*|= 2 $. 
Since $U^*$ is a proper subset of $V(D)$, $V(D) - U^* \neq \emptyset$.
 Since $U^* \neq \emptyset$ and $C^m(D)$ is connected, there exists a vertex $x$ in $V(D) - U^*$ which is adjacent to a vertex $w \in U^*$ in $C^m(D)$.
 Then, $w$ and $x$ have an $m$-step common prey $a_m$ and so
 there exists a directed $(w,a_m)$-walk of length $m$ in $D$.
 Let $a_1$ be the vertex outgoing from $w$ on this walk.
 Then $a_1 \in N^+(w) \subseteq \bigcup_{v\in U^*}N^+(v)$.
 By the choice of $U^*$, each vertex of $\bigcup_{v\in U^*}N^+(v)$ has two predators in $U^*$.
Thus there is the other predator $y$ of $a_1$ that belongs to $U^*$.
  Since $x \not\in U^*$ and $y \in U^*$,  $y$ and $x$ are distinct. Further, $y$ is an $m$-step predator of $a_m$ and so $\{w,x,y\} \subseteq N^-_m(a_m)$, which is a contradiction to Lemma~\ref{lem:triangle-free-chars}(2).
  Therefore the claim is valid.

   We prove part (1) by induction on $|U|$.
   If $U=\emptyset$, then the inequality trivially holds.
Now suppose that $|\bigcup_{v\in U}N^+(v)|\geq |U|$ for any proper vertex subset $U$ of $V(D)$ with $|U|\leq k$ for a nonnegative integer $k$ such that $k \leq |V(D)|-2$.
Take a proper subset $W$ of $V(D)$ with $k+1$ elements.
 Then $W$ is nonempty.
 Suppose, to the contrary, that $ |\bigcup_{v\in W}N^+(v)| < |W|$.
 By the above claim, there exists a vertex $w \in \bigcup_{v\in W}N^+(v) $ such that $|N^-(w) \cap W|=1$. Then $N^-(w)\cap W=\{x\}$ for some vertex $x \in W$.
 Since $x$ is the only predator of $w$ in $W$, $w \notin \bigcup_{v\in W-\{x\}}N^+(v)$.
 Then, since $w \in \bigcup_{v\in W}N^+(v)$, \[\left|\bigcup_{v\in W-\{w\}}N^+(v)\right| \leq |\bigcup_{v\in W}N^+(v)|-1.\]
 By the assumption that $|\bigcup_{v\in W}N^+(v)| < |W|$, \begin{equation}\label{eq:inequal} \left|\bigcup_{v\in W-\{w\}}N^+(v)\right| < |W|-1. \end{equation}
Yet, since $W-\{x\}$ is a proper subset of $V(D)$ with $k$ elements, by the induction hypothesis, \[\left|\bigcup_{v\in W-\{w\}}N^+(v)\right|\geq |W-\{x\}|=|W|-1,\]
 which contradicts \eqref{eq:inequal}. Therefore part (1) is true.

 To verify part (2),
 take two vertices $u$ and $v$ of $D$ and fix a positive integer $i$.
We first consider the case $N^+_i(u) \cap N^+_i(v) =V(D)$ and take a vertex $w$.
Then $w$ has at least one predator $z \in N^+_i(u) \cap N^+_i(v)$.
Therefore $w \in N^+_{i+1}(u) \cap N^+_{i+1}(v)$.
Since $w$ was arbitrarily chosen from $D$,
 $N^+_{i+1}(u) \cap N^+_{i+1}(v) =V(D)$.

Now consider the case $N^+_i(u) \cap N^+_i(v) \subsetneq V(D)$.

By part (1),
\[\left|\bigcup_{w\in N^+_i(u)\cap N^+_i(v) }N^+(w)\right| \geq \left|N^+_i(u)\cap N^+_i(v)\right|. \]
Then, since
\[\bigcup_{w\in N^+_i(u)\cap N^+_i(v) }N^+(w) \subseteq 
N^+_{i+1}(u) \cap N^+_{i+1}(v),
\]
\[|N^+_i(u)\cap N^+_i(v)| \leq |N^+_{i+1}(u) \cap N^+_{i+1}(v)|.\]
We may repeat this process until we have $|N^+_i(u)\cap N^+_i(v)|\leq |N^+_j(u) \cap N^+_j(v)|$ for any integer $j$ with $i \leq j \leq m$, which will complete the proof of part (2).
Part (3) is an immediate consequence of part (2).
\end{proof}

The inequality in Theorem~\ref{thm:triangle-free-chars}(1) is true only for a {\it proper} subset of $V(D)$ as shown by the following example.
\begin{ex} \label{ex:star}
Fix some integer $m \geq 2$.
Let $D$ be the windmill digraph with $V(D)=\{v_1,v_2,\ldots,v_m,w\}$ and $A(D)=\{(v_i,v_{i+1}) \mid 1 \leq i < m \} \cup \{(v_m,v_1)\} \cup \{(w,v_i) \mid 1 \leq i \leq m \}$. Then $C^m(D)$ is a star graph with the center $w$. However, $|N^+(V(D))| < |V(D)|$ since $w \notin N^+(V(D))$.
\end{ex}

\begin{Prop} \label{prop:preydegree}
Let $D$ be an $m$-step tree-inducing digraph with the anomaly $\alpha$.
Suppose $d^+_i(u) \geq l$ for a vertex $u$ in $D$ and positive integers $l$ and $i \leq m$.
Then the degree of $u$ is at least $l-1$ in $C^m(D)$.
Especially, if the degree of $u$ equals $l-1$ in $C^m(D)$, then $d^+_m(u)=l$ and $\alpha \in N^+_m(u)$ in $D$.
\end{Prop}

\begin{proof}
Denote by $d(u)$ the degree of a vertex $u$ in $C^m(D)$. Since $d^+_i(u) \geq l$, $d^+_m(u) \geq l$ by Theorem~\ref{thm:triangle-free-chars}(3).
Then
there are at least $l-1$ vertices in $N^+_m(u)$ each of which serves as the $m$-step common prey for exactly one pair of adjacent
vertices in $C^m(D)$ by Proposition~\ref{Prop:Helle}.
Therefore $d(u) \geq l-1$.
To show the ``especially" part, suppose, to the contrary, that $d(u)=l-1$ but $d^+_m(u) \neq l$. Then, by the hypothesis, $d^+_m(u) \geq l+1$. Thus, by the previous argument, $d(u) \geq l$, a contradiction.
Therefore $d^+_m(u) = l$.
Yet, $d(u)=l-1$, so $\alpha \in N^+_m(u)$.
\end{proof}

 \begin{Thm} \label{thm:existence-of-cycle}
Let $D$ be a tree-inducing digraph without sources.
Then each vertex lies on a directed cycle in $D$.
\end{Thm}

\begin{proof}
Suppose, to the contrary, that there exists a vertex $u$
which does not lie on any directed cycle in $D$.
Let $A, B,$ and $C$ be subsets of $V(D)$ such that
 \[A=\bigcup_{i \geq 1}N^+_i(u); \quad B=\bigcup_{i \geq 1}N^-_i(u); \quad C=V(D)-(A\cup B).\]
By the hypothesis, $N^-(u)\neq \emptyset$, so $B \neq \emptyset$.
By Lemma~\ref{lem:triangle-free-chars}(1),  $N^+(u)\neq \emptyset$, so $A \neq \emptyset$.
Since there is no directed cycle containing $u$, $A \cap B = \emptyset$.
If $u \in A $ or $u \in B$, then there exists a closed directed walk containing $u$ and so there exists a directed cycle containing $u$, which contradicts our assumption.
Thus $u \in C$ and so $C \neq \emptyset$.
We will claim the following:
\begin{equation} \label{eq:noarc1}  A \nrightarrow B, \quad A \nrightarrow C,  \quad \text{and} \quad C \nrightarrow B  \end{equation} where $X \nrightarrow Y $ for vertex sets $X$ and $Y$ of $D$ means that there is no arc from a vertex in $X$ to a vertex in $Y$.
Take three vertices $a\in A$, $b\in B$, and $c\in C$.

If there exists an arc $(a,b)$,
then a directed $(u,a)$-walk, the arc $(a,b)$, and a directed $(b,u)$-walk form a closed directed walk containing $u$ and we reach a contradiction.
If there exists an arc $(a,c)$ (resp.\ an arc $(c,b)$), then a directed $(u,a)$-walk and the arc $(a,c)$ form a directed $(u,c)$-walk (resp.\ the arc $(c,b)$ and a directed $(b,u)$-walk form a directed $(c,u)$-walk), which contradicts the choice of $c$.
Since $a$, $b$, and $c$ were arbitrarily chosen from $A$, $B$, and $C$, respectively,
the claim is valid.

By choice of the set $C$,  \begin{equation} \label{eq:noarc2}
 \{u\} \nrightarrow C, \quad \text{and} \quad C \nrightarrow \{u\}
\end{equation}
Since $D$ is a tree-inducing digraph, by Proposition~\ref{Prop:Helle}, there is a bijection between $E(C^m(D))$ and $V(D)-\{w\}$  where $w$ is the anomaly.
Then each of at least $|B|$ vertices in $B \cup \{u\}$ serves as an $m$-step common prey of a pair of adjacent vertices in $C^m(D)$.
Since $u \in C$, no vertex in $A \cup C $ can be an $m$-step predator of a vertex in $B \cup \{u\}$ by \eqref{eq:noarc1} and \eqref{eq:noarc2}.
Therefore each vertex in $B \cup \{u\}$ has an $m$-step predator only in $B$.
Consequently, we may conclude that the subgraph $H$ of $C^m(D)$ induced by $B$ has at least $|B|$ edges.
Thus $H$ contains a cycle.
Then this cycle is contained in $C^m(D)$ and we have reached a contradiction to the hypothesis that $C^m(D)$ is a tree.
\end{proof}
\begin{rem}
It is likely that, for each vertex of a digraph without sources, there is a directed cycle containing it.
However, it is not true.
For example, the digraph  given in Figure~\ref{fig:examplenocycle} has no source and no directed cycle containing the vertex $v_3$.
\begin{figure}
\begin{center}
 \begin{tikzpicture}[auto,thick]
    \tikzstyle{player}=[minimum size=5pt,inner sep=0pt,outer sep=0pt,fill,color=black, circle]
    \tikzstyle{source}=[minimum size=5pt,inner sep=0pt,outer sep=0pt,ball color=black, circle]
    \tikzstyle{arc}=[minimum size=5pt,inner sep=1pt,outer sep=1pt, font=\footnotesize]
    \path (0:0cm)     node [player]  (a) [label=left:$v_2$] {};
    \path (180:1cm)   node [player]  (b) [label=left:$v_1$] {};
    \path (0:1cm)      node [player]  (c) [label=below:$v_3$] {};
    \path (0:2cm)      node [player]  (d) [label=right:$v_4$] {};
    \path (0:3cm)      node [player]  (e) [label=right:$v_5$] {};
   \draw[black,thick,-stealth] (a) to[in=90, out=90](b);
   \draw[black,thick,-stealth] (b) to[in=270, out=270](a);
   \draw[black,thick,-stealth] (a) - +(c);
   \draw[black,thick,-stealth] (c) - +(d);
   \draw[black,thick,-stealth] (d) to[in=90, out=90](e);
   \draw[black,thick,-stealth] (e) to[in=270, out=270](d);
    \end{tikzpicture}
\caption{A digraph with no sources and no directed cycles containing $v_3$}
\label{fig:examplenocycle}
\end{center}
\end{figure}
\end{rem}

\begin{rem}
For some tree-inducing digraph $D$ with a source, Theorem~\ref{thm:existence-of-cycle} may be false. For example, the vertex $w$ given in Example~\ref{ex:star} does not lie on any directed cycle in $D$.
\end{rem}
 If an $m$-step tree-inducing digraph $D$ has a loop incident to a vertex on a directed cycle of length $2$, then $C^m(D)$ is not a star graph, which will be shown in Lemma~\ref{lem:star-not-duck}. We call such a configuration a {\it duck digraph}.
That is, a duck digraph is isomorphic to the digraph given in Figure~\ref{fig:duckgraph}.
We call the vertex with a loop in a duck digraph the {\it neck vertex }and the other one the {\it tail vertex}.
Given a digraph $D$, if $D$ contains no subdigraph isomorphic to a duck digraph, we call $D$ a {\it duck-free digraph}.

\begin{figure}
\begin{center}
    \begin{tikzpicture}[auto,thick]
    \tikzstyle{player}=[minimum size=5pt,inner sep=0pt,outer sep=0pt,fill,color=black, circle]
    \tikzstyle{source}=[minimum size=5pt,inner sep=0pt,outer sep=0pt,ball color=black, circle]
    \tikzstyle{arc}=[minimum size=5pt,inner sep=1pt,outer sep=1pt, font=\footnotesize]
    \path (0:0cm)   node [player]  (a)  [label=below:$v_1$] {};
    \path (0:1.5cm)     node [player]  (b)  [label=below:$v_2$] {};
   \draw[black,thick,-stealth] (a)to[out=45, in=135, distance=0.5cm](b);
   \draw[black,thick,-stealth] (a) to[out=100, in=150, distance=0.75cm] (a);
   \draw[black,thick,-stealth] (b) to[out=-135, in=-45,  distance=0.5cm](a);
    \end{tikzpicture}
\caption{A duck digraph with the neck vertex $v_1$ and the tail vertex $v_2$}
\label{fig:duckgraph}
\end{center}
\end{figure}

\begin{Prop} \label{prop:loopcondition}
Let $D$ be an $m$-step tree-inducing digraph 
such that (a) there exists a vertex $u$ incident to a loop in $D$ and (b) if $m=2$, then $D$ is duck-free. Then exactly one of the following statements is true.
\begin{itemize}
\item[(i)]  The vertex $u$ has exactly one predator other than $u$ and $ N^+(u)=\{u\} $.
\item[(ii)] The vertex $u$ has at least one prey other than $u$ and $ N^-(u)=\{u\} $.
\end{itemize}
  Furthermore, if (i) holds for the vertex $u$, then the vertex in $N^-(u) - \{u\}$    either is a source or is incident to a loop and (ii) holds for it.
\end{Prop}

\begin{proof}
By the condition (a), $\{u\} \subseteq N^-(u)$ and  $\{u\} \subseteq N^+(u)$.
If $N^-(u) =N^+(u)=\{u \}$, then $u$ is an isolated vertex in $C^m(D)$, which is a contradiction.
Therefore \begin{equation} \label{eq:prop:loopcondition-0-3} \{u\} \subsetneq N^-(u) \text{ or } \{u\}\subsetneq N^+(u). \end{equation}

If $N^-(u)= \{u\}$,
then the statement (i) does not hold for $u$ and, by~\eqref{eq:prop:loopcondition-0-3}, $\{u\}\subsetneq N^+(u)$ so that (ii) holds for $u$.

Now suppose that $\{u\} \subsetneq N^-(u)$.
Then (ii) does not hold for $u$.
Since $d^-(u) \le 2$ by Lemma~\ref{lem:triangle-free-chars}(2),
\[ N^-(u) =\{u,v\}\] for some vertex $v$ in $D$.
 Then \[u\in N^+(v).\]
Suppose, to the contrary, that there exists a vertex $z$ in $N^-(v) - \{u,v\}$.
Then, by using the loop incident to $u$, we may produce a directed $(u,u)$-walk, a directed $(v,u)$-walk, and a directed $(z,u)$-walk, respectively, of length $m$.
Since $m \geq 2$, $\{u,v,z\} \subseteq N^-_m(u)$, which is a contradiction to Lemma~\ref{lem:triangle-free-chars}(2).
Hence \begin{equation} \label{eq:prop:loopcondition-0-2} N^-(v) \subseteq \{u,v\}. \end{equation}
To reach a contradiction, suppose that $u \in N^-(v)$. Then \begin{equation} \label{eq:prop:loopcondition-0} \{u,v\} \subseteq N^+(u). \end{equation}
In this case, the subdigraph of $D$ induced by $\{u,v\}$ is a duck digraph.
Then, by the condition (b), $m \geq 3$.
Suppose that there exists a vertex $x$ in $N^+(v) - \{u,v\}$.
By using the loop incident to $u$, we have $\{u,v,x\} \subseteq N^+_m(v) \cap N^+_m(u)$ for each $m \geq 3$ and so, by Corollary~\ref{cor:anomaly-finder}(1), the anomaly of $D$ is contained in $N^+_m(v) \cap N^+_m(u)$.
Therefore $|N^+_m(v) \cap N^+_m(u)|=2$ by  Corollary~\ref{cor:no-more-two-common}(i), which contradicts the fact that $\{u,v,x\} \subseteq N^+_m(v) \cap N^+_m(u)$.
Thus \begin{equation} \label{eq:prop:loopcondition-0-1}u \in N^+(v) \subseteq \{u,v\}.\end{equation}
Suppose that there exists a vertex $y$ in $N^+(u)-\{u,v\}$. Then, by using the loop incident to $u$, we have $\{u,v,y\} \subseteq N^+_m(u) \cap N^+_m(v) $ and so we reach a contradiction similarly as above.
Therefore $N^+(u)=\{u,v\}$ by~\eqref{eq:prop:loopcondition-0}.
Thus, by~\eqref{eq:prop:loopcondition-0-1}, $N^+_m(u)=N^+_m(v)=\{u,v\}$, which contradicts Proposition~\ref{prop:proposition2-5}.
Hence $u \not\in N^-(v)$ and so, by~\eqref{eq:prop:loopcondition-0-2},
\begin{equation} \label{eq:prop:loopcondition}
 N^-(v) =\emptyset \quad \text{or} \quad N^-(v) =\{v\}. \end{equation}
 Then, by \eqref{eq:prop:loopcondition} and Corollary~\ref{cor:anomaly-finder}(2), $v$ is the anomaly.
 Now suppose, to the contrary, that $\{u\} \subsetneq N^+(u)$.
Take a vertex $w$ in $N^+(u)- \{u\}$.
Then $\{u,w\} \subseteq N^+_m(u) \cap N^+_m(v) $.
Since $v$ is the anomaly of $D$, $v \in N^+_m(u) \cap N^+_m(v)$ by Corollary~\ref{cor:anomaly-finder}(1).
However, $v \notin N^+_m(u)$ by \eqref{eq:prop:loopcondition}, which is a contradiction.
Hence $N^+(u)=\{u\}$ and so the statement (i) holds for $u$.
The ``furthermore" part is  true by~\eqref{eq:prop:loopcondition}.
\end{proof}

\begin{Cor} \label{cor:duckfreem>3}
Let $D$ be an $m$-step tree-inducing digraph for an integer $m \ge 3$. Then $D$ is duck-free.
\end{Cor}

\begin{proof}
Since $m \geq 3$, the condition (b) in Proposition~\ref{prop:loopcondition} is vacuously satisfied.
Suppose that there exists a vertex $v$ with a loop in $D$. Then $v$ satisfies the condition (a) in Proposition~\ref{prop:loopcondition}.
Thus, by Proposition~\ref{prop:loopcondition},
$N^+(v)=\{v\}$ or $N^-(v)=\{v\}$.
Thus $D$ is duck-free.
\end{proof}

\begin{Thm} \label{thm:existenceofoutloop}
Let $D$ be a duck-free tree-inducing digraph with a loop and without sources. Then there is a vertex $u$ with outdegree at least $2$ and $N^-(u)=\{u\}$.
Moreover, $u$ is the only one vertex with this property and $d^-(v)=2$ for each vertex $v \in N^+(u)- \{u\}$.
\end{Thm}
\begin{proof}
We note that $D$ satisfies the conditions (a) and (b) of Proposition~\ref{prop:loopcondition}.
Let $w$ be a vertex incident to a loop.
If $|N^+(w)-\{w\}| \geq 1 $, then Proposition~\ref{prop:loopcondition}(ii) holds for $w$ and so we take $w$ as $u$.
Suppose that $|N^+(w)-\{w\}|=0$.
Then Proposition~\ref{prop:loopcondition}(i) holds.
By the ``furthermore" part of Proposition~\ref{prop:loopcondition}, the vertex in $N^-(w) - \{w\}$ either is a source or is incident to a loop and Proposition~\ref{prop:loopcondition}(ii) holds for it.
Since each vertex has indegree at least $1$ by the hypothesis, the latter is true and so we take the vertex in $N^-(w) - \{w\}$ as $u$.

To show the uniqueness,
suppose that there exist two vertices $x$ and $y$ each of which has outdegree at least $2$ and indegree $1$ and is incident to a loop.
Therefore $N^-_m(x)=\{x\}$ and  $N^-_m(y)=\{y\}$.
Then, by Corollary~\ref{cor:anomaly-finder}(2), $x$ and $y$ are anomaly.
Therefore $x=y$ by Corollary~\ref{cor:no-more-two-common}(ii).
Thus $u$ is the unique vertex with $d^+(u) \geq 2 $ and $N^-(u)=\{u\}$.

Suppose, to the contrary, that $d^-(v) \neq 2$ for some vertex $v \in N^+(u)- \{u\}$.
Then $d^-(v) \leq 1$ by Lemma~\ref{lem:triangle-free-chars}(2).
Since $D$ has no source by the hypothesis, $d^-(v) \geq 1$ and so $N^-(v) = \{u\}$.
On the other hand, since $D$ is a tree-inducing digraph without sources, there exists a directed cycle $C$ containing $v$ in $D$ by Theorem~\ref{thm:existence-of-cycle}.
Since $N^-(v) = \{u\}$, $u$ lies on $C$.
Therefore there exists the $(v,u)$-section of $C$. However, since $N^-(u)=\{u\}$, there is no directed $(v,u)$-walk in $D$ and we reach a contradiction.
\end{proof}

 \section{The digraphs whose $m$-step competition graphs are star graphs}
In this section, we completely characterize the digraphs whose $m$-step competition graphs are star graphs.
The following lemma is easy to check.
\begin{Lem} \label{lem:exactlyonecycle}
For a digraph $D$, $D$ is a vertex-disjoint union of directed cycles if and only if each vertex has outdegree $1$ in $D$ and any pair of vertices has no common prey in $D$.
\end{Lem}

 \begin{Thm} \label{thm:star-source}
 An $m$-step tree-inducing digraph having a source is a windmill digraph.
 \end{Thm}

 \begin{proof}
Let $D$ be an $m$-step tree-inducing digraph having a source $v$. Then $N^-_m(v)=\emptyset$, so $v$ is the anomaly by Corollary~\ref{cor:anomaly-finder}(2).
Therefore $v$ is the only source of $D$ by Proposition~\ref{Prop:Helle}.
Thus $D$ satisfies the condition $(\text{W}1)$
for being a windmill digraph.
In addition,
\begin{equation}\label{eq:noindegreestar1} d^-_m(u)=2 \quad \end{equation}
for each vertex $u \in V(D) - \{v\}$ by Corollary~\ref{cor:no-more-two-common}(ii).

Fix $u \in V(D)- \{v\}$.
Then $d^-(u) \geq 1$ by~\eqref{eq:noindegreestar1}. By Lemma~\ref{lem:triangle-free-chars}(2), $d^-(u)\leq 2$.
Suppose, to the contrary, that $d^-(u)=1$.
Then $N^-(u)=\{x\}$ for some vertex $x$ of $D$, so $u \notin \bigcup_{v\in V(D)- \{x\}}N^+(v)$.
Since $N^-(v)=\emptyset$, $v \not\in \bigcup_{v\in V(D)- \{x\}}N^+(v)$ and so $\bigcup_{v\in V(D)- \{x\}}N^+(v)\subseteq (V(D) - \{u,v\}) $.

$
$
Therefore \[|\bigcup_{v\in V(D)- \{x\}}N^+(v))| \leq |V(D)- \{u,v\}|  < |V(D)- \{x\}|. \]
Since $V(D)- \{x\}$ is a proper subset of $V(D)$, we reach a contradiction to  Theorem~\ref{thm:triangle-free-chars}(1). Therefore
\begin{equation} \label{eq:noindegreestar} d^-(u)=2. \end{equation}

  By Lemma~\ref{lem:triangle-free-chars}(1), $d^+(u) \geq 1 $.
Suppose, to the contrary, that $d^+(u) \geq 2 $.
Then, by Theorem~\ref{thm:triangle-free-chars}(3), $d^+_{m-1}(u) \geq 2$.
On the other hand, by~\eqref{eq:noindegreestar}, $u$ is a common prey of two vertices $y$ and $y'$. Then there exist arcs $(y,u)$ and $(y',u)$ in $D$.
Take a vertex $z$ in $N^+_{m-1}(u)$.
Then there exists a directed $(u,z)$-walk $W$ of length $m-1$.
Therefore $y \rightarrow W $ is a directed $(y,z)$-walk and $y' \rightarrow W $ is a directed $(y',z)$-walk both of which have length $m$.
Thus $z \in N^+_m(y) \cap N^+_m(y')$ and so $N^+_{m-1}(u)\subseteq N^+_m(y) \cap N^+_m(y')$.
Then, since $d^+_{m-1}(u)\geq 2$, $|N^+_m(y) \cap N^+_m(y')| \geq 2$.
Therefore the anomaly $v$ must be contained in $N^+_m(y) \cap N^+_m(y')$ by
Corollary~\ref{cor:anomaly-finder}(1), which contracts the fact that $N^-(v)=\emptyset$.
Therefore \begin{equation} \label{eq:noindegreestar2}d^+(u)=1. \end{equation}
Since $u$ was arbitrarily chosen,
\eqref{eq:noindegreestar} and \eqref{eq:noindegreestar2} hold for any vertex in $V(D) - \{v\}$.
Now take two distinct vertices $x$ and $y$ in $V(D) - \{v\}$ (they exist since $D$ has at least three vertices by the definition of $m$-step tree inducing digraph). Then $d^+(x)=d^+(y)=1$.
Therefore, by Proposition~\ref{prop:proposition2-5}, $N^+(x) \cap N^+(y) = \emptyset$.
Thus, by Lemma~\ref{lem:exactlyonecycle}, $D-v$ is a vertex-disjoint union of directed cycles
and so $D$ satisfies the condition $(\text{W}2)$.
Hence \eqref{eq:noindegreestar} and \eqref{eq:noindegreestar2} deduce that
$D$ satisfies the condition $(\text{W}3)$.
\end{proof}

\begin{figure}
\begin{center}
    \begin{tikzpicture}[auto,thick]
    \tikzstyle{player}=[minimum size=5pt,inner sep=0pt,outer sep=0pt,fill,color=black, circle]
    \tikzstyle{source}=[minimum size=5pt,inner sep=0pt,outer sep=0pt,ball color=black, circle]
    \tikzstyle{arc}=[minimum size=5pt,inner sep=1pt,outer sep=1pt, font=\footnotesize]
    \path (0:0cm)   node [player]  (u)  [label=below:$u$] {};
    \path (45:1.5cm)     node [player]  (w)  [label=above:$y$] {};
    \path (0:1.5cm)     node [player]  (v)  [label=below:$z$] {};
   \draw[black,thick,-stealth] (w) to [in=45, out=225, distance=0.5cm](u);
      \draw[black,thick,-stealth] (u) to [in=180, out=90, distance=0.5cm](w);
   \draw[black,thick,-stealth] (v) to[out=180, in=0,  distance=0.5cm](u);
      \draw[black,thick,-stealth] (u) to[out=-45, in=225,  distance=0.5cm](v);
    \end{tikzpicture}
\quad \quad \quad \quad
    \begin{tikzpicture}[auto,thick]
    \tikzstyle{player}=[minimum size=5pt,inner sep=0pt,outer sep=0pt,fill,color=black, circle]
    \tikzstyle{source}=[minimum size=5pt,inner sep=0pt,outer sep=0pt,ball color=black, circle]
    \tikzstyle{arc}=[minimum size=5pt,inner sep=1pt,outer sep=1pt, font=\footnotesize]
    \path (0:0cm)   node [player]  (u)  [label=below:$u$] {};
    \path (45:1.5cm)     node [player]  (w)  [label=above:$y$] {};
    \path (0:1.5cm)     node [player]  (v)  [label=below:$z$] {};
  \draw[black,thick,-] (v) - +(w);

    \end{tikzpicture}
\caption{A digraph and its $2$-step competition graph}
\label{fig:not-digraph}
\end{center}
\end{figure}

\begin{Lem} \label{lem:star=nonsource}
Let $D$ be a digraph without sources whose $m$-step competition graph $C^m(D)$ is a star graph.
Then the following are true:
\begin{enumerate}[{(1)}]
\item There exist at most two vertices of $i$-step outdegree at least $2$ for each $1\leq i\leq m$.

\item If a vertex $v$ has a predator distinct from $v$, then $d^+(v) \leq 2$.

\item Each vertex of indegree $2$ is a prey of the center of $C^m(D)$.
\end{enumerate}
\end{Lem}
\begin{proof}
Suppose that there are three vertices $x,y,$ and $z$ having $j$-step outdegree at least $2$ for some $j \in \{1,\ldots,m\}$.
Since $C^m(D)$ is a star graph, at least two of $x,y,$ and $z$ have degree $1$ in $C^m(D)$.
Without loss of generality, we may assume that $y$ and $z$ have degree $1$ in $C^m(D)$.
Then the anomaly of $D$ is an $m$-step common prey of $y$ and $z$ by the ``especially'' part of Proposition~\ref{prop:preydegree}.
Therefore $yz$  is an edge in $C^m(D)$.
Then, since $y$ and $z$ have degree $1$ in $C^m(D)$,
$yz$ is a component in $C^m(D)$, a contradiction.
Thus part (1) is true.

To show part (2),
suppose that there exists a vertex $v$ that has a predator $v'$ distinct from $v$.
If $d^+(v) \geq 3$,
then $d^+_2(v') \geq 3$ and so, by Proposition~\ref{prop:preydegree}, $v$ and $v'$ have degree at least $2$ in $C^m(D)$, a contradiction.
Therefore $d^+(v) \leq 2$.

Suppose that there exists a vertex of indegree $2$ that is a common prey of two vertices $x$ and $y$.
Then $x$ and $y$ are adjacent in $C^m(D)$ by Theorem~\ref{thm:triangle-free-chars}(2).
Therefore $x$ or $y$ is the center of $C^m(D)$.
Thus part (3) is true.
\end{proof}

We call a directed cycle $C$ in a digraph $D$ an {\it induced directed cycle} if $C$ is an induced subdigraph of $D$.

\begin{Thm} \label{thm:stargraph-loopless}
Let $D$ be a loopless digraph whose $m$-step competition graph is a star graph.
If $D$ has no sources, then $m=2$ and $D$ is isomorphic to the digraph given in Figure~\ref{fig:loopless-stargraph}.
\end{Thm}

\begin{proof}
Suppose that $D$ has no sources.
Then, by Theorem~\ref{thm:existence-of-cycle}, $D$ has a directed cycle.
We first claim that each directed cycle in $D$ has length at least $m$.
To reach a contradiction, suppose that there exists a directed cycle $C:=v_0 \to v_1 \to \cdots \to v_{l-1} \to v_0$ of length $l \leq m-1$.
Since $D$ is loopless, $l \geq 2$ and $m \geq 3$.

Suppose that $C$ is not an induced directed cycle.
Then, since $D$ is loopless, $l \geq 3$.
Moreover, there is an arc $(v_i,v_j)$ for some $i,j \in \{0,1,\ldots, l-1\}$ so that it together with
a section of $C$ forms a directed cycle of length at most $l-1$.
Without loss of generality, we may assume that $i=0$.
Then $j \notin \{0,1\}$ and $v_j$ is a common prey of $v_0$ and $v_{j-1}$.
Accordingly, $v_j$ is a $2$-step common prey of
$v_{l-1}$ and $v_{j-2}$.
Therefore $v_0v_{j-1}$ and $v_{l-1}v_{j-2}$ are edges in $C^m(D)$ by Theorem~\ref{thm:triangle-free-chars}(2).
Thus $C^m(D)$ is not a star graph, a contradiction.
Hence $C$ is an induced directed cycle.

Suppose, to the contrary, that no vertex in $V(D) - V(C)$ has a prey in $V(C)$. Then, since $C$ is an induced directed cycle, each vertex on $C$ has exactly one $m$-step predator in $D$.
Therefore each vertex on $C$ is the anomaly by Corollary~\ref{cor:anomaly-finder}(2).
Since $l\geq 2$,
we reach a contradiction to the uniqueness of the anomaly.
Therefore there exists a vertex $a$ in $V(D) - V(C)$ that has a prey on $C$.
Without loss of generality, we may assume that $v_0$ is a prey of $a$.
Therefore $v_0$ is a common prey of $a$ and $v_{l-1}$ and so, by Theorem~\ref{thm:triangle-free-chars}(2), $av_{l-1}$ is an edge in $C^m(D)$.
Thus $a$ or $v_{l-1}$ is the center of $C^m(D)$.
On the other hand, since $a$ is not source, $a$ has a predator $b$.
To show $b\neq v_{l-2}$,
suppose $b=v_{l-2}$.
Then $\{a,v_{l-1} \} \subseteq N^+(v_{l-2})$.
Therefore $d^+_2(v_{l-3})\geq 2$ and $d^+_3(v_{l-4}) \geq 2$ (we assume that each subscript of the vertices on $C$ is reduced to modulo $l$).
Thus each of $v_{l-2}, v_{l-3},$ and $v_{l-4}$ has an $m$-step outdegree at least $2$ by Theorem~\ref{thm:triangle-free-chars}(3).
Hence $v_{l-2}=v_{l-4}$ by Lemma~\ref{lem:star=nonsource}(1) and so $l=2$.
Then we can check that $d^+(v_0)\geq 2$, $d^+_2(v_1) \geq 2$, and $d^+_2(a) \geq 2$.
Therefore
each of $v_0$, $v_1$, and $a$ has an $m$-step outdegree at least $2$ by Theorem~\ref{thm:triangle-free-chars}(3),
which contradicts Lemma~\ref{lem:star=nonsource}(1).
Thus
\[ b \neq v_{l-2}.\]
If $b$ is distinct from $v_{l-1}$, then $v_{l-2}$ and $b$ are adjacent since $v_0$ is a $2$-step common prey of $v_{l-2}$ and $b$, a contradiction to the fact that $a$ or $v_{l-1}$ is the center of $C^m(D)$.
Therefore $b=v_{l-1}.$
Thus $v_0$ is a $3$-step common prey of $v_{l-2}$ and $v_{l-3}$ and so, by Theorem~\ref{thm:triangle-free-chars}(2),  $v_{l-2}v_{l-3}$ is an edge in $C^m(D)$.
Hence $v_{l-2}$ or $v_{l-3}$ is the center of $C^m(D)$.
Then, since $v_{l-2} \neq v_{l-1}$, and $a$ or $v_{l-1}$ is the center of $C^m(D)$,
$v_{l-1}(=v_{l-3})$ is the center of $C^m(D)$.
Therefore $l=2$.
Thus $v_0 \to v_1 \to a \to v_0$ and $a \to v_0 \to v_1 \to v_0$ and so, by Theorem~\ref{thm:triangle-free-chars}(2), $v_0a$ is an edge in $C^m(D)$, which contradicts the fact that $v_1$ is the center of $C^m(D)$.

Hence we have shown that
\begin{itemize}
\item[($\ast$)] each directed cycle in $D$ has length at least $m$.
\end{itemize}
Since $D$ is loopless and has no sources,
\begin{equation}\label{eq:prop:stargraph-loopless-1}
d^+(v)  \leq 2
\end{equation}
for each vertex $v$ in $D$ by Lemma~\ref{lem:star=nonsource}(2).
If each vertex has outdegree $1$, then each vertex has indegree $1$ since $D$ has no sources, and so $C^m(D)$ is edgeless.
Therefore there exists a vertex $u$ of outdegree at least $2$. Thus
\begin{equation}\label{eq:prop:stargraph-loopless-2}
d^+(u)=2
\end{equation}
by~\eqref{eq:prop:stargraph-loopless-1}.
Since $D$ has no sources,
there exists a directed walk
\begin{equation}\label{eq:prop:stargraph-loopless-2-1}
W:=x \to y \to u \end{equation}
in $D$.
Suppose, to the contrary, that $m \geq 3$. Then, by ($\ast$), $x$, $y$, and $u$ are distinct.
Since $d^+(u)=2$,
$d^+_3(x)  \geq 2$, and $d^+_2(y) \geq 2$.
Therefore each of $x, y, u$ has $m$-step outdegree at least $2$ by Theorem~\ref{thm:triangle-free-chars}(3),
which contradicts Lemma~\ref{lem:star=nonsource}(1).
Thus $m \leq 2$ and so \[m=2.\]
Let $c$ be the center of $C^2(D)$.
Since  $d^+_2(v)\leq 4$ for each vertex $v$ in $D$ by~\eqref{eq:prop:stargraph-loopless-1},
$c$ has degree at most $4$ in $C^2(D)$ by Lemma~\ref{lem:triangle-free-chars}(2) and so $|V(D)| \leq 5$.
Since $|V(D)| > m=2$,  $|V(D)| \in \{3,4,5\}$.

Suppose $|V(D)|=3$.
Then, since $u \neq y$,
$V(D)=\{u,y,z\}$.
By Theorem~\ref{thm:existence-of-cycle},
there exists a directed cycle in $D$.
We take a longest directed cycle $C$ of length $l$.
Then, since $D$ is loopless and $|V(D)|=3$, $l=2$ or $3$.
If $l=2$, then, by Theorem~\ref{thm:existence-of-cycle},
$D$ is isomorphic to the digraph given in Figure~\ref{fig:not-digraph} and so $C^2(D)$ has an isolated vertex, a contradiction.
Thus $l=3$.
Then $C=u \to z \to y \to u$ or $u \to y \to z  \to u$.
We note that $N^+(u)=\{y,z\}$ and, by \eqref{eq:prop:stargraph-loopless-2-1}, $y \to u$.
To show $C=u \to z \to y \to u$, suppose $C=u \to y \to z  \to u$. 
Then $u$ is a common prey of $y$ and $z$ and $z$ is a common prey of $u$ and $y$, so, by Theorem~\ref{thm:triangle-free-chars}(2), $yz,uy$ are edges in $C^2(D)$.
Moreover, $z$ is a $2$-step common prey of $u$ and $z$, and so $uz$ is an edge in $C^2(D)$.
Thus $C^2(D)$ is a triangle,  which contradicts the fact $C^2(D)$ is a star. 
Therefore $C=u \to z \to y \to u$.
Since $u$ has outdegree $2$ by~\eqref{eq:prop:stargraph-loopless-2},
$u\to y$.
Therefore we obtain a subdigraph isomorphic to the one given in Figure~\ref{fig:loopless-stargraph}.
Thus $y$ is a common prey of $u$ and $z$ and $y$ is a $2$-step common prey of $u$ and $y$.
Hence $uz$ and $uy$ are edges in $C^2(D)$ by Theorem~\ref{thm:triangle-free-chars}(2).
Now it is easy to check that adding more arcs to the digraph given in Figure~\ref{fig:loopless-stargraph} results in the edge joining $y$ and $z$ in $C^2(D)$.
Therefore we conclude that $D$ is isomorphic to the one given in Figure~\ref{fig:loopless-stargraph}.

Now suppose that $|V(D)|=4$ or $5$.
Then $c$ has at least three $2$-step prey and so
\[d^+(c)=2\]
by~\eqref{eq:prop:stargraph-loopless-1}.
Let $u_1$ and $u_2$ be the prey of $c$.
If each of $u_1$ and $u_2$ has outdegree $1$, then $c$ has at most two $2$-step prey, which is impossible.
Therefore at least one of them has outdegree $2$.
Without loss of generality,
we may assume that
$u_1$ has outdegree $2$.
Then $c$ and $u_1$ are the only vertices of outdegree $2$ by Lemma~\ref{lem:star=nonsource}(1).
Hence $u_2$ has outdegree $1$ by \eqref{eq:prop:stargraph-loopless-1}.
Moreover, $c$ has at most three $2$-step prey and so $c$ has degree at most $3$ in $C^2(D)$.
Therefore $|V(D)| = 4$.
Thus $c$ has degree $3$ in $C^2(D)$.

By Lemma~\ref{lem:triangle-free-chars}(2),
each vertex is a $2$-step common prey of at most two vertices.
Therefore $c$ has exactly three $2$-step prey in $D$.
Let $x$ be one of them.
Then, other than $c$, there is exactly one $2$-step predator of $x$.
We denote it by $\tilde{x}$.
Then, for distinct $2$-step prey $x$ and $y$ of $c$, $\tilde{x} \neq \tilde{y}$.
Suppose that $u_1$ has indegree $1$.
If some prey $d$ of $u_1$ has indegree $1$, then $d$ is a $2$-step prey of $c$ and $N^-_2(d)=\{c\}$, which contradicts the existence of $\tilde{d}$.
Therefore each prey of $u_1$ has indegree $2$.
Thus, by Lemma~\ref{lem:star=nonsource}(3),
each prey of $u_1$ is a prey of $c$.
Since $D$ is loopless and $u_1$ has outdegree $2$, $d^+(c) \geq 3$, a contradiction.
Thus $u_1$ has indegree $2$ and $|N^-(u_1)-\{c\}|=1$.
Then
the vertex in $ N^-(u_1)- \{c\}$
is a $2$-step common predator of the two prey $w$ and $z$ of $u_1$.
Now, even if $w \neq z$, $\tilde{w}=\tilde{z}$, a contradiction.
Hence the statement is true.
\end{proof}

\begin{Lem} \label{lem:star-generating-type-B}
Let $D$ be a windmill digraph or an $m$-conveyor digraph.
Then $C^m(D)$ is a star graph.
\end{Lem}

\begin{proof}
We suppose that $D$ is a windmill digraph with the source $v$.
Then $v$ and another vertex $w$ have a common prey by ($\text{W}2$) and ($\text{W}3$).
Therefore, by ($\text{W}2$), $v$ and $w$ have an $m$-step common prey for any $m \geq 1$ and so $v$ and $w$ are adjacent in $C^m(D)$.
By ($\text{W}1$) and ($\text{W}2$), any two vertices other than $v$ cannot have an $m$-step common prey for any $m \geq 1$.
Thus $C^m(D)$ is a star graph with the center $v$.

Now we suppose that $D$ is an $m$-conveyor digraph with the loop $v$ satisfying (M1) and (M2).
Thus any vertex $w$ other than $v$ has a unique $m$-step prey $x$ on a directed cycle containing $w$ and $w$ is the only $m$-step predator of $x$ in $V(D)- \{v\}$, and so $w$ is not adjacent to any vertex belonging to $V(D)- \{v\}$ in $C^m(D)$.
Since $D$ is a weakly connected digraph, by (M2), each directed cycle in $D-v$ has a vertex that is a prey of $v$ and so there exists an internally secure lane $W$ in $D$ containing $x$.
By (M3), $W$ has length at most $m$.
Since $v$ is incident to a loop by (M1), we may obtain a directed walk of length $m$ from $v$ to $x$ by using the loop incident to $v$.
Hence $x$ is an $m$-step common prey of $v$ and $w$ and so $v$ and $w$ are adjacent in $C^m(D)$, which implies that $C^m(D)$ is a star graph with the center $v$.
\end{proof}

\begin{Lem} \label{lem:star-not-duck}
Let $D$ be a tree-inducing digraph whose $m$-step competition graph is a star graph.
Then $D$ is duck-free.
\end{Lem}
\begin{proof}
Suppose, to the contrary, that $D$ contains a subdigraph $H$ isomorphic to a duck digraph (see Figure~\ref{fig:duckgraph} for an illustration).
 Then, by Corollary~\ref{cor:duckfreem>3}, $m=2$.
Let $v_1$ and $v_2$ be the neck vertex and the tail vertex, respectively, of $H$.
By the definition of a duck digraph, $\{(v_1,v_1),$ $(v_1,v_2),(v_2,v_1)\} \subseteq A(D)$.
It is easy to check that $\{v_1,v_2\} \subseteq N^-_2(v_1)$ and $\{v_1,v_2\} \subseteq N^-_2(v_2)$.
By Lemma~\ref{lem:triangle-free-chars}(2),
\[N^-_2(v_1)=N^-_2(v_2)=\{v_1,v_2\}.\]
 Since a predator of $v_1$ or $v_2$ would belong to $N^-_2(v_1)$,
 \begin{equation} \label{eq:prop:double-chainnotstar-0} N^-(v_1)=\{v_1,v_2\} \quad \text{and}
 \quad \{v_1\}\subseteq N^-(v_2)\subseteq \{v_1,v_2\}. \end{equation}
To show $N^+(v_1)=\{v_1,v_2\}$ by contradiction,
suppose that there exists a vertex $v_3$ distinct from $v_1$ and $v_2$ in $N^+(v_1)$. Then $\{v_1,v_2,v_3\} \subseteq N^+_2(v_1) \cap N^+_2(v_2)$.
Therefore one of $v_1,v_2,v_3$ is the anomaly of $D$ by Corollary~\ref{cor:anomaly-finder}(1).
Thus $|N^+_2(v_1) \cap N^+_2(v_2)|=2$ by Corollary~\ref{cor:no-more-two-common}(i), which contradicts the fact that $\{v_1,v_2,v_3\} \subseteq N^+_2(v_1) \cap N^+_2(v_2)$.
Hence \begin{equation} \label{eq:prop:double-chainnotstar-1}N^+(v_1)=\{v_1,v_2\}. \end{equation}
If $N^+(v_2) \subseteq \{v_1,v_2\}$,
then $H$ is a component of $C^m(D)$  by~\eqref{eq:prop:double-chainnotstar-0}  and~\eqref{eq:prop:double-chainnotstar-1}, which contradicts the hypothesis that $C^2(D)$ is a star graph with at least three vertices (a tree-inducing digraph has at least three vertices by definition).
Therefore
there exists a vertex $v_3$ in $N^+(v_2) -\{v_1,v_2\}$.
If $v_3$ is incident to a loop,
then $v_1$, $v_2$, and $v_3$ are $2$-step predators of $v_3$, which contradicts Lemma~\ref{lem:triangle-free-chars}(2).
Therefore $v_3$ is not incident to a loop.
Moreover, $v_3$ has outdegree at least $1$ by Lemma~\ref{lem:triangle-free-chars}(1).
Then, by~\eqref{eq:prop:double-chainnotstar-0},
neither $v_1$ nor $v_2$ can be a prey of $v_3$, so
there must be a vertex $v_4$ in $N^+(v_3) -\{v_1,v_2,v_3\}$.
Therefore $\{v_1,v_2,v_4\} \subseteq N^+_2(v_2)$ and $\{v_1,v_2,v_3\} \subseteq N^+_2(v_1)$.
Thus each degree of $v_1$ and $v_2$ is at least $2$ in $C^2(D)$ by Proposition~\ref{prop:preydegree}.
Hence $C^2(D)$ is not a star graph, a contradiction.
\end{proof}

Now we are ready to prove Theorem~\ref{thm:complete_star}.

\begin{proof}[Proof of Theorem~\ref{thm:complete_star}]
To show the ``only if" part, suppose that there exists a digraph $D$ with $n$ vertices whose $m$-step competition graph is a star graph for some $2\leq  m < n$.
Then $D$ is duck-free by Lemma~\ref{lem:star-not-duck}.
If $D$ has a source, then $D$ is a windmill digraph by Theorem~\ref{thm:star-source}.
Suppose that $D$ has no sources.
If $D$ is loopless, then (iii) is true by Theorem~\ref{thm:stargraph-loopless}.
Now we suppose that $D$ has a loop.
We will show that $D$ is an $m$-conveyor digraph.
Since $D$ has a loop and $D$ is duck-free,
there exists a vertex $v$ such that
$N^-(v)=\{v\}$
and $d^+(v)\geq 2$ by Theorem~\ref{thm:existenceofoutloop}.
Since $N^-(v)=\{v\}$,
(M1) is satisfied and $v$ is the only $m$-step predator of $v$.
Then, by Corollary~\ref{cor:anomaly-finder}(2), $v$ is the anomaly.

To reach a contradiction, suppose that there exists a vertex $w$ distinct from $v$ having outdegree at least $2$.
Then $d^+_i(w) \geq 2$ for each $1\leq i\leq m$ by Theorem~\ref{thm:triangle-free-chars}(3).
If $w$ has degree $1$ in $C^m(D)$, then
$v \in N^+_m(w)$ by the ``especially" part of Proposition~\ref{prop:preydegree}, which contradicts the fact that $v$ is the only $m$-step predator of $v$.
Therefore
$w$ has degree at least $2$ in $C^m(D)$ and so $w$ is the center of $C^m(D)$.
Since $D$ has no sources,
$w$ has a predator $x$.
Since $N^+_{m-1}(w)\subseteq N^+_m(x)$,
$x$ has at least two $m$-step prey each of which is not $v$.
Then, since $v$ is the anomaly, $x$ has degree at least $2$ by Proposition~\ref{prop:preydegree}, and so $x$ is the center of $C^m(D)$.
Thus $x=w$.
Since $x$ was arbitrarily taken, $N^-(w)=\{w\}$.
Consequently, $w$ is the anomaly by Corollary~\ref{cor:anomaly-finder}(2).
Then, since $v\neq w$, we reach a contradiction to the uniqueness of the anomaly.
Therefore $v$ is the only vertex of outdegree at least $2$ in $D$ and so, by Lemma~\ref{lem:triangle-free-chars}(1), each vertex in $V(D) - \{v\}$ has outdegree $1$.
Thus any pair of vertices in $V(D) - \{v\}$ has no common prey by Proposition~\ref{prop:proposition2-5}.
Hence $D-v$ is a vertex-disjoint union of directed cycles by Lemma~\ref{lem:exactlyonecycle} and so (M2) is satisfied.

Suppose that there exists an internally secure lane $W$ of length at least $m+1$.
Then the $m$th interior vertex $v'$ on $W$ has exactly one $m$-step predator in $D$. Thus $v'$ is the anomaly in $D$ by Corollary~\ref{cor:anomaly-finder}(2).
However, since $N^-(v)=\{v\}$, we obtain $v'\neq v$ and so we reach a contradiction to the uniqueness of the anomaly.
Therefore each internally secure lane of $D$ has length at most $m$ and so (M3) is satisfied. 
Thus
$D$ is an $m$-conveyor digraph.
Hence the ``only if" part is true.

Now we show the `if" part.
If $D$ is a windmill digraph or an $m$-conveyor digraph,
then $C^m(D)$ is a star graph by Lemma~\ref{lem:star-generating-type-B}.
In addition, it is easy to check that the $2$-step competition graph of the digraph given in Figure~\ref{fig:loopless-stargraph} is a star graph.
Therefore the ``if" part is true and so this completes the proof.
\end{proof}
\section{Acknowledgement}
This research was supported by the National Research Foundation of Korea(NRF) funded by the Korea government(MSIP) (NRF-2017R1E1A1A03070489, NRF-2022R1A2C1009648, and 2016R1A5A1008055).

\end{document}